\documentclass{amsart}
\pdfoutput=1
\usepackage[utf8]{inputenc}
\usepackage[T1]{fontenc}
\usepackage{lmodern}
\usepackage{mathtools}
\usepackage[lite]{amsrefs}
\usepackage{microtype}
\usepackage{MnSymbol}
\usepackage{enumitem} % advanced enumerate, which handles label and ref
\setlist[enumerate,1]{label=\textup{(\roman*)}}% ensure enumerates in theorems are upright

\usepackage[pdftitle={Local cyclic homology for nonarchimedean Banach algebras},
pdfauthor={Ralf Meyer and Devarshi Mukherjee},
pdfsubject={Mathematics}
]{hyperref}

\BibSpec{book}{%
  +{}  {\PrintPrimary}                {transition}
  +{,} { \textit}                     {title}
  +{.} { }                            {part}
  +{:} { \textit}                     {subtitle}
  +{,} { \PrintEdition}               {edition}
  +{}  { \PrintEditorsB}              {editor}
  +{,} { \PrintTranslatorsC}          {translator}
  +{,} { \PrintContributions}         {contribution}
  +{,} { }                            {series}
  +{,} { \voltext}                    {volume}
  +{,} { }                            {publisher}
  +{,} { }                            {organization}
  +{,} { }                            {address}
  +{,} { \PrintDateB}                 {date}
  +{,} { }                            {status}
  +{}  { \parenthesize}               {language}
  +{}  { \PrintTranslation}           {translation}
  +{;} { \PrintReprint}               {reprint}
  +{.} { }                            {note}
  +{.} {}                             {transition}
  +{} { \PrintDOI}                   {doi}
  +{} { available at \url}            {eprint}
  +{}  {\SentenceSpace \PrintReviews} {review}
}

\renewcommand*{\PrintDOI}[1]{\href{http://dx.doi.org/\detokenize{#1}}{doi: \detokenize{#1}}}

\theoremstyle{plain}
\newtheorem{theorem}{Theorem}[section]
\newtheorem{lemma}[theorem]{Lemma}
\newtheorem{corollary}[theorem]{Corollary}
\newtheorem{proposition}[theorem]{Proposition}
\theoremstyle{remark}
\newtheorem{remark}[theorem]{Remark}
\theoremstyle{definition}
\newtheorem{definition}[theorem]{Definition}
\newtheorem{example}[theorem]{Example}
\numberwithin{theorem}{section}

\newcommand\C{\mathbb C}
\newcommand\N{\mathbb N}
\newcommand\Q{\mathbb Q}
\newcommand\Z{\mathbb Z}

\newcommand{\coma}{\widehat}% pi-adic completion
\newcommand{\comb}{\overbracket[.7pt][1.4pt]}% bornological completion
\newcommand{\updagger}{\textup{\tiny\!\!\dagger}}% upright dagger also in italics

\newcommand{\diff}{\mathrm{d}}% differential
\newcommand{\defeq}{\mathrel{:=}} % per Definition
\newcommand*{\into}{\rightarrowtail}
\newcommand*{\onto}{\twoheadrightarrow}

\newcommand*{\ling}[1]{#1_\mathrm{lg}}% linear growth bornology

\DeclarePairedDelimiterX{\setgiven}[2]{\{}{\}}{#1\,{:}\,\mathopen{}#2}% set given by

\newcommand\hot{\mathbin{\comb{\otimes}}}% bornologically complete tp
\newcommand\haotimes{\mathbin{\coma{\otimes}}}% adically complete tp

\DeclareMathOperator{\coker}{coker}
\DeclareMathOperator{\HA}{HA}% analytic cyclic homology
\DeclareMathOperator{\HAC}{\mathbb{HA}}% analytic cyclic homology complex
\DeclareMathOperator{\HLC}{\mathbb{HL}}% analytic cyclic homology complex
\DeclareMathOperator{\HP}{HP}% periodic cyclic homology

\newcommand{\ev}{\mathrm{ev}}% evaluation homomorphism
\newcommand{\nb}{\nobreakdash}% non-breaking hyphen

\newcommand{\dvr}{V}% discrete valuation ring
\newcommand{\dvgen}{\pi}% uniformiser
\newcommand{\dvf}{F}% field of fractions of \dvr
\newcommand{\resf}{\mathbb F}% residue field of \dvr

\begin{document}
\title[Nonarchimedean local cyclic homology]
{Local cyclic homology for\\ nonarchimedean Banach algebras}

\author{Ralf Meyer}
\email{rmeyer2@uni-goettingen.de}

\author{Devarshi Mukherjee}
\email{devarshi.mukherjee@mathematik.uni-goettingen.de}

\address{Mathematisches Institut\\
  Universit\"at Göttingen\\
  Bunsenstra\ss{}e 3--5\\
  37073 Göttingen\\
  Germany}

\begin{abstract}
  Let~\(\dvr\) be a complete discrete valuation ring with
  uniformiser~$\dvgen$.  We introduce an invariant of Banach
  $\dvr$\nb-algebras called local cyclic homology.  This invariant
  is related to analytic cyclic homology for complete,
  bornologically torsionfree $\dvr$\nb-algebras.  It is shown that
  local cyclic homology only depends on the reduction mod~$\dvgen$
  of a Banach $\dvr$\nb-algebra and that it is homotopy invariant,
  matricially stable, and excisive.
\end{abstract}
\subjclass[2020]{19D55}

\maketitle

\section{Introduction}
\label{sec:introduction}

This article is part of a programme to define analytic cyclic
homology theories for bornological algebras over nonarchimedean
fields, with the goal of defining well behaved homology theories
for algebras over their residue fields, which may have finite
characteristic
(see~\cites{Cortinas-Cuntz-Meyer-Tamme:Nonarchimedean,
  Meyer-Mukherjee:Bornological_tf, Cortinas-Meyer-Mukherjee:NAHA,
  Meyer-Mukherjee:HA}).  In this article, we define a version of
cyclic homology that gives good results for Banach algebras over
nonarchimedean local fields such as~\(\Q_p\).  Our prototype here is
the local cyclic homology theory in the archimedean case, defined
first by Puschnigg for inductive systems of ``nice'' Fr\'echet
algebras over~\(\C\) (see \cite{Puschnigg:Diffeotopy}) and then
simplified by the first author in the setting of complete
bornological algebras over~\(\C\) (see~\cite{Meyer:HLHA}).  The
remarkable property of local cyclic homology for algebras
over~\(\C\) is that it remains well behaved for
\(\mathrm{C}^*\)\nb-algebras.  For instance, it is invariant under
homotopies that are merely continuous, and it is stable under the
\(\mathrm{C}^*\)\nb-algebraic tensor product with the compact
operators.

The definition of archimedean local cyclic homology has two key
ingredients.  The first one is analytic cyclic homology for complete
bornological algebras over~\(\C\).  The first author arrived at its
definition by rewriting Connes' entire cyclic cohomology in terms of
bornologies and taking the most natural ``predual'' of that
cohomology theory.  The second key ingredient is to turn a Banach
algebra into a bornological algebra using the bornology of
precompact subsets.  This allows to use locally defined linear maps,
which map a continuous function to a nearby smooth function.

Before we can talk about the nonarchimedean version of this, we fix
some notation.  Let~\(\dvr\) be a complete discrete valuation ring.
Let~\(\dvgen\) be its uniformiser, \(\resf\) its residue field,
and~\(\dvf\) its fraction field.  We assume throughout that~\(\dvf\)
has characteristic zero.  The nonarchimedean version of analytic
cyclic homology has already been defined
in~\cite{Cortinas-Meyer-Mukherjee:NAHA}.  It is a functor~\(\HA_*\)
from the category of complete, torsionfree bornological
$\dvr$\nb-algebras to the category of \(\dvf\)\nb-vector spaces.  We
usually work with the enriched version~\(\HAC\) of~\(\HA_*\), which
takes values in the derived category of the quasi-Abelian category
of countable projective systems of inductive systems of Banach
\(\dvf\)\nb-vector spaces; to define \(\HA_*(A)\), we first take the
homotopy limit of \(\HAC(A)\), then a colimit.  This gives a chain
complex of \(\dvf\)\nb-vector spaces, whose homology is \(\HA_*(A)\).

In the archimedean case, the functors \(\HAC\) and \(\HA_*\) are
only homotopy invariant for smooth homotopies, and they are not
expected to behave well for \(\mathrm{C}^*\)\nb-algebras.  The
nonarchimedean versions of \(\HAC\) and \(\HA_*\) are shown
in~\cite{Cortinas-Meyer-Mukherjee:NAHA} to be invariant under dagger
homotopies, which take values in the completed tensor product with
the algebra \(\dvr[t]^\updagger\).  The latter algebra is the
Monsky--Washnitzer algebra of the affine plane.  It consists of
those power series \(\sum c_n t^n\) with \(c_n \in \dvr\) where the
valuations of the coefficients~\(c_n\) grow at least linearly.  It
is unclear whether \(\HAC\) and \(\HA_*\) behave well for larger
algebras such as the \(\dvgen\)\nb-adic completion
\(\coma{\dvr[t]}\), which is defined by asking only for
\(\lim c_n = 0\).

In this article, we define a nonarchimedean analogue of the
precompact bornology, namely, the compactoid bornology.  A
subset~\(S\) of a bornological \(\dvr\)\nb-module~\(M\) is called
\emph{compactoid} if there is a bounded \(\dvr\)\nb-submodule
\(T\subseteq M\) such that, for every \(n \in \N\), there is a
finite set \(F_n \subseteq T\) with
\(S \subseteq \dvr F_n + \dvgen^n\cdot T\).  Let~\(B\) be a Banach
algebra over~\(\dvf\) with a submultiplicative norm.  Then its unit
ball \(D\subseteq B\) is an algebra over~\(\dvr\) with some nice
extra properties, which make it a dagger algebra.  For any dagger
algebra~\(D\), let~\(D'\) be~\(D\) with the compactoid bornology.
We define the local cyclic homology \(\HLC(D)\) as
\(\HAC(D')\).  We are going to prove that this theory has good
homological properties and gives good
results for many interesting Banach \(\dvr\)\nb-algebras.  
Namely, local cyclic homology is invariant under continuous
homotopy, tensoring with the \(\dvgen\)\nb-adic completion of finite
matrices, and satisfies excision for all extensions of dagger
algebras.  And we compute it for Banach algebra versions of Leavitt
path algebras, Laurent polynomials in several variables, and the
Tate algebras of curves over~\(\dvr\).  

The hardest part of the work was already done in our previous
paper~\cite{Meyer-Mukherjee:HA}, where we built an analytic cyclic
homology theory for algebras over the residue field~\(\resf\).  A
key result in~\cite{Meyer-Mukherjee:HA} says that \(\HAC(A)\) for an
algebra~\(A\) over~\(\resf\) is naturally isomorphic to \(\HAC(D)\)
if~\(D\) is a dagger algebra with \(D/\dvgen D \cong A\) and such
that the quotient bornology on \(D/\dvgen D\) is the fine one; we
briefly call~$D$ fine mod~$\dvgen$.  Roughly speaking, any dagger
algebra lifting of~\(A\) that is also fine mod~\(\dvgen\) may be
used to compute \(\HAC(A)\).  The compactoid bornology is always
fine mod~\(\dvgen\), and we prove that a dagger algebra with the
compactoid bornology remains a dagger algebra.  Therefore,
\(\HAC(A) \cong \HLC(D)\) for any dagger algebra~\(D\) with
\(D/\dvgen D \cong A\).  Here~\(D\) may be \(\dvgen\)\nb-adically
complete or, in other words, a Banach \(\dvr\)\nb-algebra.

We could also have defined the local cyclic homology for dagger
algebras by \(\HLC(D) \defeq \HAC(D/\dvgen D)\).  This is equivalent
to our definition using the precompact bornology.  This definition
looks very quick, and we use it implicitly to prove the formal
properties of local cyclic homology and compute some examples, based
on the results in~\cite{Meyer-Mukherjee:HA} showing that analytic
cyclic homology for algebras over~\(\resf\) is polynomially homotopy
invariant, stable with respect to the \(\resf\)\nb-algebra of finite
matrices, and excisive, and based on examples computed there.

While the definition of \(\HLC(D)\) as \(\HAC(D/\dvgen D)\) looks
rather quick, it just shifts all the difficulties to the analytic
cyclic homology for \(\resf\)\nb-algebras.  This is defined by
lifting an \(\resf\)\nb-algebra to a projective system of inductive
systems of \(\dvr\)\nb-algebras and then taking an analytic cyclic
homology complex for the latter object.  Results
in~\cite{Meyer-Mukherjee:HA} allow to replace the lifting that is
used to define \(\HAC(D/\dvgen D)\) by a pro-dagger algebra.  The
choices of such liftings that the general theory
in~\cite{Meyer-Mukherjee:HA} provides are, however, still rather
unwieldy.  The main observation in this article is that we may
choose the given Banach algebra~\(D\) with its precompact bornology
to compute \(\HAC(D/\dvgen D)\).  Thus our definition of \(\HLC(D)\)
turns out to be more direct.  In addition, it clarifies the
similarity with local cyclic homology for bornological
\(\C\)\nb-algebras.

We may also prove the formal properties of \(\HLC\) using that
\(\HLC(D_1) \cong \HLC(D_2)\) if \(D_1\subseteq D_2\) is a dagger
subalgebra with \(D_1/\dvgen D_1 = D_2/\dvgen D_2\).  For instance,
the homotopy invariance of \(\HLC\) reduces to the homotopy
invariance of analytic cyclic homology for dagger homotopies that is
proven in~\cite{Cortinas-Meyer-Mukherjee:NAHA}.  This proof would be
very close to the proof in~\cite{Meyer:HLHA} that local cyclic
homology for bornological \(\C\)\nb-algebras is homotopy invariant.
In the end, we switched to proofs that reduce to
\(\resf\)\nb-algebras because these are shorter than proofs that
remain within the realm of dagger algebras.

We end this article with an illuminating counterexample.  A key open
question in the study of analytic cyclic homology for dagger
algebras is when the analytic and periodic cyclic homology of a nice
dagger algebra agree.  This is important because periodic cyclic
homology is shown
in~\cites{Cortinas-Cuntz-Meyer-Tamme:Nonarchimedean} to specialise
to rigid cohomology for Monsky--Washnitzer algebras, whereas
analytic cyclic homology is shown in~\cite{Meyer-Mukherjee:HA} to
depend only on the reduction mod~\(\dvgen\).  In the archimedean
case, Khalkhali~\cite{Khalkhali:Connections_entire} proved that
entire and periodic cyclic cohomology are isomorphic for Banach
\(\C\)\nb-algebras of finite projective dimension.  This may lead to
the hope that finite projective dimension may suffice to show that
nonarchimedean analytic and periodic cyclic homology are isomorphic.
We prove that
this is not the case, by showing that the Tate algebra
\(\coma{\dvr[t]}\) with the compactoid bornology is quasi-free
--~which means of projective dimension~\(1\).  But its periodic and
analytic cyclic homology differ.

The paper is structured as follows.  We begin by recalling some
basic definitions and results about bornologies and dagger algebras
in Section~\ref{sec:prelim}.  We define the compactoid bornology and
nuclearity in Section~\ref{sec:compactoid_bornologies}.
Section~\ref{sec:inheritance} deals with inheritance properties of
the compactoid bornology.  We define local cyclic homology in
Section~\ref{sec:def_HL}.  We prove that it has the expected
homological properties in Section~\ref{sec:formal_properties}.  We
compute the local cyclic homology of some examples in
Section~\ref{sec:computations}.

This article got started by a remark by Joachim Cuntz, and we thank
him for several useful discussions.

\section{Preliminaries}
\label{sec:prelim}

Let~\(\dvr\) be a complete discrete valuation ring.  Let~\(\dvgen\)
be its uniformiser, \(\resf\) its residue field, and~\(\dvf\) its
fraction field.  We assume throughout that~\(\dvf\) has
characteristic zero.

As in~\cites{Cortinas-Cuntz-Meyer-Tamme:Nonarchimedean,
  Meyer-Mukherjee:Bornological_tf,
  Cortinas-Meyer-Mukherjee:NAHA,Meyer-Mukherjee:HA}, we use the
framework of bornologies to do nonarchimedean analysis.
A \emph{bornology} on a set~\(X\) is a collection of its subsets,
which are called \emph{bounded subsets}, such
that finite subsets are bounded and finite unions and subsets of
bounded subsets remain bounded.

A \emph{bornological \(\dvr\)\nb-module} is a
\(\dvr\)\nb-module~\(M\) with a bornology such that every bounded
subset is contained in a bounded \(\dvr\)\nb-submodule.  We call a
\(\dvr\)\nb-module map \(f \colon M \to N\) \emph{bounded} if it
maps bounded subsets of~\(M\) to bounded subsets of~\(N\).  A
bornological \(\dvr\)\nb-algebra is a bornological
\(\dvr\)\nb-module with a bounded multiplication map.  A
\emph{complete} bornological \(\dvr\)\nb-module is a bornological
\(\dvr\)\nb-module in which every bounded subset is contained in a
bounded, \(\dvgen\)\nb-adically complete \(\dvr\)\nb-submodule.
Every bornological \(\dvr\)\nb-module~$M$ has a
completion~$\comb{M}$ (see
\cite{Cortinas-Cuntz-Meyer-Tamme:Nonarchimedean}*{Proposition~2.14}).

\begin{example}
  The most basic example of a bornology on a \(\dvr\)\nb-module is
  the \emph{fine bornology}, which consists of those subsets that
  are contained in a finitely generated \(\dvr\)\nb-submodule.  Any
  fine bornological \(\dvr\)\nb-module is complete.  By default, we
  equip modules over the residue field~\(\resf\) with the fine
  bornology.
\end{example}

\begin{definition}[\cite{Meyer-Mukherjee:Bornological_tf}*{Definition~4.1}]
  \label{def:bornologically_tf}
  We call a bornological \(\dvr\)\nb-module~\(M\) (bornologically)
  \emph{torsionfree} if multiplication by~\(\dvgen\) is a
  bornological embedding, that is, $M$ is algebraically torsionfree
  and
  \(\dvgen^{-1} \cdot S \defeq \setgiven{x \in M}{\dvgen x \in S}\)
  is bounded for every bounded subset \(S \subseteq M\).  A
  \(\dvr\)\nb-module with the fine bornology is bornologically
  torsionfree if and only if it is torsionfree in the purely
  algebraic sense.  For the rest of this article, we briefly write
  ``torsionfree'' instead of ``bornologically torsionfree''.
\end{definition}

\begin{definition}[\cites{Cortinas-Cuntz-Meyer-Tamme:Nonarchimedean,
    Meyer-Mukherjee:Bornological_tf}]
  \label{def:dagger_algebra}
  We call a bornological \(\dvr\)\nb-algebra~\(D\)
  \emph{semidagger} if, for every bounded subset \(S \subseteq D\),
  the \(\dvr\)\nb-submodule \(\sum_{i=0}^\infty \dvgen^i S^{i+1}\)
  is bounded in~\(D\).  A complete, torsionfree, semidagger
  bornological \(\dvr\)\nb-algebra is called a \emph{dagger
    algebra}.
\end{definition}

\begin{example}
  \label{exa:resf_semidagger}
  Any \(\resf\)\nb-algebra with the fine bornology is semidagger and
  complete.
\end{example}

\begin{example}
  \label{exa:Banach_algebra}
  Let~\(B\) be a Banach \(\dvf\)\nb-algebra.  We assume the norm
  of~\(B\) to be submultiplicative.  Let \(D\subseteq B\) be the
  unit ball.  Then \(D\cdot D\subseteq D\), and~\(D\) becomes a
  \(\dvgen\)\nb-adically complete, torsionfree \(\dvr\)\nb-algebra.
  Conversely, if such an algebra~\(D\) is given, then
  \(D\hookrightarrow D\otimes \dvf\) and there is a unique norm on
  \(D\otimes \dvf\) with unit ball~\(D\).

  Let~\(D\) be the unit
  ball of a Banach \(\dvf\)\nb-algebra as above.  Then we call~\(D\)
  with the bornology where all subsets are bounded a \emph{Banach
    \(\dvr\)\nb-algebra}.  This bornology makes~\(D\) a dagger
  algebra.
\end{example}

\begin{definition}
  Any bornology on a \(\dvr\)\nb-algebra~\(D\) is contained in a
  smallest semidagger bornology, namely, the bornology generated by the
  \(\dvr\)\nb-submodules of the form
  \(\sum_{i=0}^\infty \dvgen^i S^{i+1}\), where \(S \subseteq D\) is
  bounded in the original bornology.  This is called the
  \emph{linear growth bornology}.  We denote~\(D\) with the linear
  growth bornology by~\(\ling{D}\).
\end{definition}

If~\(D\) is torsionfree, then the completion
\(D^\dagger \defeq \comb{\ling{D}}\) is a dagger algebra (see
\cite{Meyer-Mukherjee:Bornological_tf}*{Proposition~3.8} or, in
slightly different notation,
\cite{Cortinas-Cuntz-Meyer-Tamme:Nonarchimedean}*{Lemma 3.1.12}).

\begin{definition}
  A bornological \(\dvr\)\nb-module~\(M\) is called \emph{fine
    mod~\(\dvgen\)} if the quotient bornology on \(M/\dvgen M\) is
  the fine one.  Equivalently, any bounded subset is contained in
  \(F+\dvgen M\) for a finitely generated \(\dvr\)\nb-submodule
  \(F\subseteq M\).
\end{definition}

\section{Compactoid bornologies}
\label{sec:compactoid_bornologies}

Let~\(A\) be an algebra over the residue field and let~\(D\) be a
dagger algebra with \(D/\dvgen D \cong A\).  If, in addition, \(D\)
is fine mod~\(\dvgen\), then the main result
of~\cite{Meyer-Mukherjee:HA} gives a natural quasi-isomorphism
\(\HAC(D) \cong \HAC(A)\).  It seems likely, however, that this
fails when~\(D\) is a Banach \(\dvr\)\nb-algebra as in
Example~\ref{exa:Banach_algebra}.  We are going to introduce a
nonarchimedean analogue of the precompact bornology on a topological
\(\mathbb{C}\)\nb-vector space, which makes any dagger algebra fine
mod~\(\dvgen\).  This is the basis for our definition of local
cyclic homology, exactly as in the archimedean case
in~\cite{Meyer:HLHA}.

\begin{definition}
  \label{def:compactoid}
  Let~\(M\) be a bornological \(\dvr\)\nb-module.  A subset
  \(S\subseteq M\) is \emph{compactoid} if it is contained in a bounded 
  \(\dvr\)\nb-submodule \(T\subseteq M\) such that, for every
  \(n \in \N\), there is a finite set \(F_n \subseteq T\) with
  \(S \subseteq \dvr F_n + \dvgen^n\cdot T\).

  We call~\(M\)
  \emph{nuclear} if any bounded subset of~\(M\) is already
  compactoid.
\end{definition}

\begin{remark}
  \label{rem:Schneider-compactoid}
  The analogous concept of a compactoid \(\dvr\)\nb-submodule of a
  locally convex \(\dvf\)\nb-vector space has already been studied
  by Peter Schneider (see
  \cite{Schneider:Nonarchimedean}*{Section~12}).
\end{remark}

Compactoid subsets are bounded, and the compactoid subsets in~\(M\)
form another bornology.  This bornology is always fine
mod~\(\dvgen\).

\begin{example}
  Let~\(M\) be a Banach \(\dvr\)\nb-module as in
  Example~\ref{exa:Banach_algebra}.  Then \(S\subseteq M\) is
  compactoid if and only if for every \(n \in \N\), the image
  of~\(S\) in \(M/\dvgen^n M\) is finitely generated.  This is the
  largest torsionfree bornology on~\(M\) that is fine
  mod~\(\dvgen\).
\end{example}

\begin{example}
  \label{exa:torsion-compactoid}
  For any \(\dvr\)\nb-module~\(M\), its fine bornology is
  nuclear.
\end{example}

The following proposition characterises the compactoid bornology in
a different way, which is analogous to a very useful description of
the precompact bornology on a Fr\'echet space over~\(\C\).

\begin{proposition}
  \label{prop:equivalent_compactoid}
  Let~\(M\) be a torsionfree bornological \(\dvr\)\nb-module and
  let~\(S\) be a \(\dvr\)\nb-submodule.  The following are
  equivalent:
  \begin{enumerate}
  \item \label{char_1}%
    \(S \subseteq M\) is compactoid;
  \item \label{char_2}%
    there is a bounded
    \(\dvr\)\nb-submodule~\(T\) containing~\(S\) such that for
    each~\(n\), the \(\dvr\)\nb-submodule of \(T/\dvgen^n T\)
    generated by the image of~\(S\) is finitely generated;
  \item \label{char_4}%
    there are a bounded \(\dvr\)\nb-submodule \(T\subseteq M\)
    and a null sequence~\((t_n)\) in~\(T\) with
    \[
      S = \setgiven*{ s =\sum_{n = 0}^\infty  c_n t_n}{(c_n) \in
        \ell^\infty(\N, \dvr) \text{ and } s \text{ converges in } T};
    \]
  \item \label{char_3}%
    there are a bounded \(\dvr\)\nb-submodule \(T\subseteq M\) with
    \(S\subseteq T\) and a null sequence~\((t_k)\) in~\(T\) such
    that \(S \subseteq \sum_{k=0}^\infty \dvr t_k + \dvgen^n T\) for
    each \(n \in \N\).
  \end{enumerate}
\end{proposition}

\begin{proof}
  The equivalence between \ref{char_1} and~\ref{char_2} is trivial.
  It is easy to prove that \ref{char_4} implies~\ref{char_3}.  It
  remains to show that \ref{char_2} implies~\ref{char_4} and that
  \ref{char_3} implies~\ref{char_2}.  Assume~\ref{char_3} and
  let~\(T\) and~\((t_k)\) be as in~\ref{char_3}.  For fixed
  \(n\in\N\), there is \(k_0\in\N\) with \(t_k \in \dvgen^n T\) for
  \(k> k_0\).  Then
  \(S \subseteq \sum_{k=0}^{k_0} \dvr t_k + \dvgen^n T\).
  Thus~\ref{char_3} implies~\ref{char_2}.  Conversely,
  assume~\ref{char_2}, and let \(S\) and~\(T\) be as
  in~\ref{char_2}.  We construct a null sequence~\((t_k)\) as
  in~\ref{char_4} inductively, by choosing an increasing sequence
  \((k_n)_{n\in\N}\) and
  \(t_{k_n+1},\dotsc, t_{k_{n+1}}\in S\cap \dvgen^n T\) for
  all~\(n\), such that if $s\in S$ is written as
  $s \equiv \sum_{i=0}^{k_n} c_i t_i \bmod \pi^n T$, then there are
  $c_i$ for $i=k_n+1,\dotsc,k_{n+1}$ with
  $s \equiv \sum_{i=0}^{k_{n+1}} c_i t_i \bmod \pi^{n+1} T$.  Since
  \(S\subseteq T\), we may start the induction with \(k_0=-1\).  Let
  \(\mathsf{Im}_n(S)\) be the image of~\(S\) in \(T/\dvgen^n T\).
  In the induction step from~\(n\) to~\(n+1\), we use that
  \(\mathsf{Im}_{n+1}(S)\) is finitely generated.  Then so is
  \(\bigl(\sum_{i=0}^{k_n} \dvr t_i + \mathsf{Im}_{n+1}(S)\bigr)
  \cap \dvgen^n T /\dvgen^{n+1} T\) because~\(\dvr\) is Noetherian.
  Let \(t_{k_n+1}, \dotsc, t_{k_{n+1}} \in \dvgen^n T\) be
  representatives of generators for this submodule.  Let $s\in S$.
  Then there are coefficients~$c_i$ for $0\le i \le k_n$ with
  $s \equiv \sum_{i=0}^{k_n} c_i t_i \bmod \pi^n T$.  Then
  $s - \sum_{i=0}^{k_n} c_i t_i \in \bigl(\sum_{i=0}^{k_n} \dvr t_i
  + \mathsf{Im}_{n+1}(S)\bigr) \cap \dvgen^n T /\dvgen^{n+1} T$ may
  be written as $\sum_{i=k_n+1}^{k_{n+1}} c_i t_i$.  Then
  $s \equiv \sum_{i=0}^{k_{n+1}} c_i t_i \bmod \pi^{n+1} T$.  The
  construction ensures that the infinite series
  $\sum_{i=0}^\infty c_i t_i$ converges $\dvgen$\nb-adically in~$T$
  towards~$s$.
\end{proof}

\begin{definition}
  Let~\(M'\) denote~\(M\) with the compactoid bornology.
\end{definition}

\begin{proposition}
  \label{pro:compactoid-nuclear}
  Let~\(M\) be a torsionfree bornological \(\dvr\)\nb-module.  Then
  the bornological \(\dvr\)\nb-module~\(M'\) is nuclear.
\end{proposition}

\begin{proof}
  Let \(S\subseteq M'\) be bounded.  That is, \(S\) is compactoid in
  the bornology of~\(M\).
  Proposition~\ref{prop:equivalent_compactoid} provides a bounded
  \(\dvr\)\nb-submodule \(T\subseteq M\) and a null
  sequence~\((x_n)\) in~\(T\) such that
  \[
    S = \setgiven*{ s =\sum_{n = 0}^\infty  c_n x_n}{(c_n) \in
      \ell^\infty(\N, \dvr) \text{ and } s \text{ converges in } T}.
  \]
  Since~\((x_n)\) is a null sequence in~\(T\), there are
  \(a_n \in \N\) and \(y_n\in T\) with \(x_n = \dvgen^{a_n} y_n\),
  \(\lim a_n = \infty\) and \(\lim y_n = 0\) in~\(T\).  The subset
  \[
    S' = \setgiven*{ s =\sum_{n=0}^\infty c_n y_n}{ (c_n) \in
      \ell^\infty(\N, \dvr), s \text{ converges in } T}
  \]
  is still compactoid in~\(T\) by
  Proposition~\ref{prop:equivalent_compactoid}.  We claim that~\(S\)
  is compactoid as a subset of~\(S'\).  For the proof, we show that any
  \(\dvgen\)\nb-adically convergent series
  \(S \ni s = \sum_{n = 0}^\infty c_n x_n\) is
  \(\dvgen\)\nb-adically convergent in~\(S'\).  That~\(s\)
  converges in~\(T\) means that for each \(e \geq 1\), there is
  an~\(n_0\) such that
  \(s - \sum_{n=0}^{n_1} c_n x_n \in \dvgen^e T\) for all
  \(n_1 \geq n_0\).  Therefore, there are \(\beta_{e, n_1} \in T \)
  with
  \[
    \dvgen^e \beta_{e, n_1}
    = s - \sum_{n=0}^{n_1} c_n x_n
    = \sum_{n = n_1 + 1}^\infty c_n x_n
    = \sum_{n= n_1 + 1}^\infty c_n \dvgen^{a_n} y_n.
  \]
  If~\(n_1\) is big enough, then $a_n \ge e$ for all $n>n_1$.
  Since~\(M\) is torsionfree, we get
  \[
    \beta_{e, n_1}
    = \sum_{n = n_1 + 1}^\infty c_n \dvgen^{a_n - e}  y_n.
  \]
  This series converges in~\(T\) because \(\beta_{f,n_2}\) exist for
  \(n_2 \gg n_1\).  Then its limit~\(\beta_{e, n_1}\) belongs
  to~\(S'\).  Therefore, the series \(\sum_{n = 0}^\infty c_n x_n\)
  converges in~\(S'\) as required.
\end{proof}

\section{Inheritance properties of the compactoid bornology}
\label{sec:inheritance}

Recall that~\(M'\) denotes~\(M\) with the compactoid bornology.
We are going to prove that~$M'$ inherits many properties from~$M$
and that the property of being nuclear is preserved by several
constructions with bornologies.  We will use many of these results
in our study of local cyclic homology.  Let~\(M\) be a bornological
\(\dvr\)\nb-module.

\begin{lemma}
  \label{lem:compactoid_complete}
  If the bornological \(\dvr\)\nb-module~\(M\) is complete, then so
  is~\(M'\).
\end{lemma}

\begin{proof}
  Let~\(S\) be a compactoid subset of~\(M\).  Then there is a
  bounded \(\dvr\)\nb-submodule \(T \subseteq M\) such that for any
  \(m \in \N\), there is a finite subset \(F_m\subseteq T\) with
  \(S \subseteq \dvr F_m + \dvgen^m T\).  Since~\(M\) is complete,
  we may choose~\(T\) to be a bounded \(\dvgen\)\nb-adically
  complete \(\dvr\)\nb-submodule.  Let
  \(T' \defeq \bigcap_{m \in \N} (\dvr F_m + \dvgen^m T)\).  Then
  \(S \subseteq T'\) and~\(T'\) is \(\dvgen\)\nb-adically complete
  because it is \(\dvgen\)\nb-adically closed in~\(T\).  It is
  compactoid by construction.  Thus~\(M'\) is complete.
\end{proof}

\begin{lemma}
  \label{lem:compactoid_torsionfree}
  If the bornological \(\dvr\)\nb-module~\(M\) is torsionfree, then
  so is~\(M'\).
\end{lemma}

\begin{proof}
  Let \(S \subseteq M\) be a compactoid subset.  Let \(T\)
  and~\(F_n\) be as in the definition of being compactoid and let
  \(n\ge1\).  Since~\(M\) is torsionfree,
  \(\dvgen^{-1} T \defeq \setgiven{x \in M}{\dvgen x \in T}\) is
  bounded.  By construction,
  \begin{equation}
    \label{eq:compactoid_torsionfree_proof}
    \dvgen^{-1}S
    \subseteq \dvgen^{-1} \dvr F_{n+1} + \dvgen^n T
    \subseteq \dvgen^{-1} \dvr F_{n+1} + \dvgen^n \dvgen^{-1} T.
  \end{equation}
  Multiplication by~\(\dvgen\) on~\(M\) is injective because~\(M\)
  is torsionfree.  It identifies \(\dvgen^{-1} \dvr F_{n+1}\) with a
  submodule of~\(\dvr F_{n+1}\).  Since~\(\dvr\) is Noetherian, it
  follows that \(\dvgen^{-1} \dvr F_{n+1}\) is again finitely
  generated.  Therefore, \eqref{eq:compactoid_torsionfree_proof}
  witnesses that~\(\dvgen^{-1} S\) is compactoid.
\end{proof}

\begin{lemma}
  \label{lem:compactoid_adjoint_functor}
  On torsionfree bornological \(\dvf\)\nb-modules, the assignment
  \(M \mapsto M'\) is the right adjoint functor of the inclusion of
  the subcategory of nuclear, torsionfree bornological
  \(\dvr\)\nb-modules.
\end{lemma}

\begin{proof}
  We assume~\(N\) to be torsionfree to ensure that~\(N'\) is nuclear
  by Proposition~\ref{pro:compactoid-nuclear}.
  Lemma~\ref{lem:compactoid_torsionfree} shows that~\(N'\) is again
  torsionfree.
  Let \(M\) and~\(N\) be torsionfree bornological
  \(\dvr\)\nb-modules.  Assume that the bornology on~\(M\) is
  nuclear.  We claim that a \(\dvr\)\nb-linear map
  \(\varphi \colon M \to N\) is bounded if and only if it is bounded
  as a map to~\(N'\).  One implication is trivial because compactoid
  subsets are bounded.  Let \(S \subseteq M\) be a bounded subset.
  By assumption, \(S\) is compactoid.  That is, there is a bounded
  \(\dvr\)\nb-submodule \(T \subseteq M\) containing~\(S\), such
  that for each~\(n\), there is a finite subset \(F_n\subseteq T\)
  with \(S \subseteq \dvr F_n + \dvgen^n T\).  The inclusions
  \(\varphi(S) \subseteq \dvr \varphi(F_n) + \dvgen^n \varphi(T)\)
  witness that \(\varphi(S)\subseteq N\) is compactoid.  That is,
  \(\varphi\) is bounded as a map \(M \to N'\).
\end{proof}

\begin{lemma}
  \label{lem:completion_compactoid}
  If~\(M\) is a nuclear bornological \(\dvr\)\nb-module, then
  so is~\(\comb{M}\).
\end{lemma}

\begin{proof}
  Let \(S \subseteq \comb{M}\) be a bounded subset.  The inductive
  limit description of~$\comb{M}$ shows that~$S$ is contained in the
  image of the \(\dvgen\)\nb-adic completion~\(\coma{U}\) of some
  bounded submodule \(U \subseteq M\).  Since~\(M\) is compactoid,
  there is a bounded submodule \(U' \subseteq M\) such that the
  image of~$U$ in $U'/\dvgen^k U'$ is finitely generated for each
  $k\in\N$.  This is equal to the image of~\(\coma{U}\) in
  $\coma{U'}/\dvgen^k \coma{U'} \cong U'/\dvgen^k U'$.  This
  witnesses that~$S$ is compactoid.
\end{proof}

We now investigate the behaviour of the compactoid bornology under
tensor products and extensions.

\begin{proposition}
  \label{prop:compactoid_tensor-product}
  Let \(M\) and~\(N\) be complete, torsionfree bornological
  \(\dvr\)-modules.  Then there is an isomorphism of complete
  bornological \(\dvr\)-modules
  \[
    M' \hot N' \cong (M \hot N)'.
  \]
  If $M$ and~$N$ are nuclear, then so is $M\hot N$.
\end{proposition}

\begin{proof}
  Lemma~\ref{lem:compactoid_complete} shows that \(M'\) and~\(N'\)
  are complete.  A submodule in \(M' \hot N'\) is bounded if it is
  in the image of \(S \haotimes T\), where \(S\) and~\(T\) are
  \(\dvgen\)\nb-adically complete, compactoid submodules of \(M\)
  and~\(N\).  By definition, this means that there are bounded,
  \(\dvgen\)\nb-adically complete \(\dvr\)\nb-submodules \(S'\)
  and~\(T'\) of \(M\) and~\(N\) such that the images of~\(S\)
  in~\(S' \onto S'/\dvgen^n S'\) and of~\(T\) in
  \(T' \onto T'/\dvgen^n T'\) are finitely generated for each~\(n\).
  Then the image of \(S \haotimes T\) in
  \(S' \haotimes T' /\dvgen^n (S' \haotimes T')\) is finitely
  generated for each~\(n\).  This says that \(S \haotimes T\) is
  compactoid in \(M\otimes N\).

  Conversely, let \(X\subseteq M \hot N\) be a compactoid subset of
  \(M \hot N\).  Then there is a bounded \(\dvgen\)-adically
  complete submodule \(U \subseteq M \hot N\), such that the image
  of~\(X\) under the quotient map to \(U/\dvgen^n U\) is finitely
  generated for each~\(n\).  There are bounded
  \(\dvgen\)\nb-adically complete submodules \(S'\subseteq M\) and
  \(T'\subseteq N\) such that~\(U\) is contained in the image
  of~\(S'\haotimes T'\) in \(M\hot N\).  We have
  \(S'\haotimes T'/ \dvgen^n (S'\haotimes T') \cong (S'\otimes
  T')/\dvgen^n (S'\otimes T')\), and the image of~\(X\) remains
  finitely generated in \((S'\otimes T')/\dvgen^n (S'\otimes T')\)
  for all \(n\in\N\).  By
  Proposition~\ref{prop:equivalent_compactoid}, there is a null
  sequence \((x_k)_{k \in \N}\) in \(S' \haotimes T'\) such
  that~\(X\) is contained in the image of
  \(\sum_{k=0}^\infty \dvr x_k + \dvgen^n (S' \haotimes T')\) in
  \(M\hot N\).  Since \(\lim x_k = 0\), there are \(l_k \to \infty\)
  with \(x_k \in \dvgen^{2 l_k} S' \haotimes T'\).  Since
  \(S'\haotimes T'/ \dvgen^n (S'\haotimes T') \cong (S'\otimes
  T')/\dvgen^n (S'\otimes T')\), the proof of
  Proposition~\ref{prop:equivalent_compactoid} shows that we may
  arrange \(x_k \in \dvgen^{2 l_k} S' \otimes T'\).  Then
  \(x_k = \dvgen^{2 l_k} \sum_{i = 1}^{m_i} a_i \otimes b_i\) for
  some \(m_i \in \N\), \(a_i \in S'\) and \(b_i \in T'\) for
  all~\(i\).  Let \(S\) and~\(T\) be the \(\dvgen\)\nb-adically
  closed \(\dvr\)\nb-submodules in \(S'\) and~\(T'\)
  generated by the bounded subsets \(\{\dvgen^{l_k} a_i\}\) and
  \(\{\dvgen^{l_k} b_i\}\), respectively.  By
  Proposition~\ref{prop:equivalent_compactoid}, \(S\) and~\(T\) are
  compactoid submodules of \(S'\) and~\(T'\), and
  \(X \subseteq S \haotimes T\).  That is, \(X\) is bounded in
  \(M' \hot N'\).

  By definition, a bornological $\dvr$\nb-module~$M$ is nuclear if
  and only if the canonical map $M'\to M$ is a bornological isomorphism.
  If $M'=M$ and $N'=N$, then $(M\hot N)' = M'\hot N' = M \hot N$ as
  well.
\end{proof}

\begin{proposition}
  \label{prop:exactness}
  Let \(K \into L \onto P\) be an extension of complete, torsionfree
  bornological \(\dvr\)\nb-modules and let~\(P\) be algebraically
  torsionfree.  Then \(K' \into L' \onto P'\) is an extension of
  complete bornological \(\dvr\)\nb-modules as well.  In addition,
  $L$ is nuclear if and only if both $K$ and~$P$ are nuclear.
\end{proposition}

\begin{proof}
  The functoriality of the compactoid bornology
  (Lemma~\ref{lem:compactoid_adjoint_functor}) shows that a subset
  \(S\subseteq K\) that is compactoid in~\(K\) remains compactoid
  in~\(L\).  Conversely, we claim that if \(S \subseteq K\) is
  compactoid as a subset of~\(L\), it is compactoid as a subset
  of~\(K\).  Let \(T\supseteq S\) be a bounded subset of~\(L\) such
  that the image of~\(S\) in \(T/\dvgen^n T\) is finitely generated
  for all \(n\in\N\).  Since \(S\subseteq K\), it follows that
  \(S\subseteq T\cap K\).  Since~\(P\) is torsionfree,
  \(\dvgen^n x\in K\) for \(x\in L\) can only happen if \(x\in K\).
  This says that the canonical map
  \((T\cap K)/ \dvgen^n (T\cap K) \to T/\dvgen^n T\) is injective.
  Therefore, the image of~\(S\) in \((T\cap K)/ \dvgen^n (T\cap K)\)
  is finitely generated for all \(n\in\N\).  This finishes the proof
  of the claim that the inclusion
  \(K' \to L'\) is a bornological embedding.

  The quotient map clearly remains bounded as a map
  \(q\colon L' \to P'\).  It remains to prove that it is a
  bornological quotient map, that is, any compactoid subset of~\(P\)
  is the image of a compactoid subset of~\(L\).

  Let \(S\subseteq P\) be compactoid.  Since~$P$ is complete,
  Proposition~\ref{prop:equivalent_compactoid} implies that there is
  a bounded, \(\dvgen\)\nb-adically complete \(\dvr\)\nb-submodule
  \(T \subseteq P\) containing~\(S\) and a null sequence~\((x_n)\)
  in~$T$ such that
  \(S \subseteq \setgiven{s = \sum_{n \in \N} c_n x_n}{(c_n) \in
    \ell^\infty(\N, \dvr)}\); these sequences converge automatically
  because \(\lim x_n=0\) and~\(T\) is complete.  Since~\(L\) is
  complete, there is a bounded, \(\dvgen\)\nb-adically complete
  \(\dvr\)\nb-submodule \(W\subseteq L\) that lifts~$T$.  We may
  lift~\((x_n)\) to a null sequence~\((\tilde{x}_n)\) in~\(W\).
  Then
  \(\tilde{S} = \setgiven{\sum_{n=0}^\infty c_n
    \tilde{x}_n}{(c_n) \in \ell^\infty(\N, \dvr)}\) is a compactoid
  submodule of~\(L\) that lifts~\(S\).

  By definition, a bornological $\dvr$\nb-module~$M$ is nuclear if
  and only if the canonical map $M'\to M$ is a bornological isomorphism.
  The category of bornological $\dvr$\nb-modules is quasi-Abelian.
  This implies a version of the Five Lemma.  For our two extensions
  \(K' \into L' \onto P'\) and \(K \into L \onto P\), it says that
  the map $L' \to L$ is a bornological isomorphism if the maps
  $K' \to K$ and $P' \to P$ are bornological isomorphisms.
  Conversely, it is easy to see that subspaces and quotients inherit
  nuclearity.
\end{proof}

\begin{lemma}
  \label{lem:compactoid_semidagger}
  Let~\(D\) be a bornological algebra and let~\(D'\) be~\(D\) with
  the compactoid bornology.  This is again a bornological algebra.
  If~\(D\) is semidagger, then so is~\(D'\).
\end{lemma}

\begin{proof}
  It is easy to see that the product of compactoid subsets is again
  compactoid, so that~\(D'\) is a bornological algebra.  Now
  assume~\(D\) to be semidagger.  We are going to prove that~\(D'\)
  is semidagger as well.  Let~\(S\) be a compactoid
  \(\dvr\)\nb-submodule of~\(D\).  We must prove that
  \(S_1 \defeq \sum_{n=0}^\infty \dvgen^n S^{n+1}\) is compactoid as
  well.  Since~\(S\) is compactoid, there are a bounded subset~\(T\)
  and a null sequence \((x_n)_{n \in \N}\) in \(T\) such that
  \(S \subseteq \sum_{n \in \N} \dvr x_n + \dvgen^k T\) for each
  \(k \in\N\).  Since~\(D\) is semidagger, the subset
  \(T_{1/2} \defeq \sum_{i=0}^\infty \dvgen^{\lceil i/2\rceil}
  T^{i+1}\supseteq T\) is still bounded
  (see~\cite{Cortinas-Cuntz-Meyer-Tamme:Nonarchimedean}*{Lemma
    3.1.10}).  For \(n\in\N\), \(i = (i_0,\dotsc,i_n)\in\N^{n+1}\),
  let \(x_i \defeq x_{i_0} \dotsm x_{i_n} \in T^n\).  Then
  \[
    S_1 = \sum_{n=0}^\infty \dvgen^n S^{n+1}
    \subseteq \sum_{n=0}^\infty \sum_{i \in \N^{n+1}} \dvr \dvgen^n x_i
    + \dvgen^k \dvgen^n T^{n+1}
    \subseteq \biggl(\sum_{n=0}^\infty \sum_{i \in\N^{n+1}}
    \dvr \dvgen^n x_i\biggr) + \dvgen^k T_{1/2}.
  \]
  We may arrange the countable set of elements~\((\dvgen^n x_i)\)
  for \(n\in\N\), \(i\in \N^{n+1}\) into a sequence that converges
  to zero in~\(T_{1/2}\).  Therefore, the inclusion above witnesses
  that~\(S_1\) is compactoid.
\end{proof}

\begin{lemma}
  \label{lem:compactoid_dagger}
  If~\(D\) is a dagger algebra, then so is~\(D'\).
\end{lemma}

\begin{proof}
  This follows from Lemmas \ref{lem:compactoid_complete},
  \ref{lem:compactoid_torsionfree}
  and~\ref{lem:compactoid_semidagger}.
\end{proof}

\begin{proposition}
  \label{prop:compactoid-dagger}
  Let~\(R\) be a nuclear bornological \(\dvr\)\nb-algebra.  Then
  \(\ling{R}\) and \(R^\updagger = \comb{\ling{R}}\) are nuclear.
\end{proposition}

\begin{proof}
  By Lemma~\ref{lem:completion_compactoid}, we only need to prove
  that~\(\ling{R}\) is nuclear.  Let~\(S\) be a subset of linear
  growth in~\(R\).  Then there is a bounded submodule
  \(T \subseteq R\) with
  \(S \subseteq T_1 \defeq \sum_{j=0}^\infty \dvgen^j T^{j+1}\).
  Since~\(R\) is compactoid, Proposition
  \ref{prop:equivalent_compactoid} gives a bounded
  submodule~\(\tilde{T}\) of~\(R\) and a null sequence
  \((x_n) \in \tilde{T}\) such that
  \(T \subseteq \sum_{n \in \N} \dvr x_n + \dvgen^k \tilde{T}\) for
  each \(k \geq 1\).  Now we consider the countable set of all
  products \(\dvgen^j x_{i_1} \dotsm x_{i_{j+1}}\) for
  \(j,i_1,\dotsc,i_{j+1}\in\N\).  For each~\(k\), there are only
  finitely many such terms that do not belong
  to~\(\dvgen^k \tilde{T}\).  Therefore, we may order these products
  into a sequence~\((y_m)_{m\in\N}\) that converges to~\(0\) in the
  \(\dvgen\)\nb-adic topology on~\(\tilde{T}\).  Now it follows from
  the construction that
  \(T_1\subseteq \sum_{m \in \N} \dvr y_m + \dvgen^k \tilde{T}\)
  for each \(k\in \N\).  Thus~\(T_1\)
  is compactoid.  Hence so is~\(S\).
\end{proof}

\begin{corollary}
  \label{cor:Monsky-Washnitzer-compactoid}
  Let~\(R\) be a torsionfree bornological algebra with the fine
  bornology.  Then \(R^\updagger = \comb{\ling{R}}\) is nuclear.
\end{corollary}

\begin{proof}
  The fine bornology on~\(R\) is nuclear by
  Example~\ref{exa:torsion-compactoid}.  Then apply
  Proposition~\ref{prop:compactoid-dagger}.
\end{proof}
 
Let~\(R\) be any torsionfree bornological \(\dvr\)\nb-algebra.
There is a canonical bounded algebra homomorphism
\(R \to R^\updagger\).  Since the compactoid bornology is
functorial, it induces a bounded algebra homomorphism
\(R' \to (R^\updagger)'\).  The target is a dagger algebra by
Lemma~\ref{lem:compactoid_dagger}.  So this homomorphism extends
uniquely to a bounded homomorphism
\[
  (R')^\updagger \to (R^\updagger)'.
\]
In general, the homomorphism \((R')^\updagger \to (R^\updagger)'\)
need not be invertible.  The following proposition describes the two
situations where we know this:

\begin{proposition}
  \label{prop:linear_growth_compactoid-commutes}
  Let~\(R\) be a torsionfree bornological algebra.  If~\(R\) is
  nuclear or dagger, then the canonical map
  \((R')^\updagger \to (R^\updagger)'\) is an isomorphism.
\end{proposition}

\begin{proof}
  Suppose first that~\(R\) is nuclear.  Then \(R \cong R'\), so
  that \((R')^\updagger = R^\updagger\).  The canonical map
  \((R^\updagger)' \to R^\updagger\) generates the required inverse
  map \((R^\updagger)' \to (R')^\updagger\).  Next assume that~\(R\)
  is a dagger algebra.  Then \(R \cong R^\updagger\) and so
  \((R^\updagger)' \cong R'\).  The canonical map
  \(R' \to (R')^\updagger\) provides the required inverse map
  \((R^\updagger)' \to (R')^\updagger\).
\end{proof}

\section{Definition of local cyclic homology}
\label{sec:def_HL}

Our definition of local cyclic homology is based on the
\emph{analytic} cyclic homology theories introduced in
\cites{Cortinas-Meyer-Mukherjee:NAHA, Meyer-Mukherjee:HA}.  We do
not recall how they are defined.  Suffice it to say the following.
Let \(\overleftarrow{\mathsf{Ind}(\mathsf{Ban}_\dvf)}\) be the
quasi-Abelian category of countable projective systems of inductive
systems of Banach \(\dvf\)\nb-vector spaces.  There are functors
\begin{align*}
  \HAC \colon \bigl\{ \text{complete, torsionfree bornological
    $\dvr$-algebras}  \bigr\}
  &\to
  \mathsf{Der}\Bigl(\overleftarrow{\mathsf{Ind}(\mathsf{Ban}_\dvf)}\Bigr),\\
  \HAC \colon \bigl\{ \resf\text{-algebras}  \bigr\} &\to
  \mathsf{Der}\Bigl(\overleftarrow{\mathsf{Ind}(\mathsf{Ban}_\dvf)}\Bigr).
\end{align*}
Their definitions are based on the Cuntz--Quillen approach to
cyclic homology.  The first functor is invariant under
dagger homotopies and Morita equivalences and satisfies excision for
extensions with a bounded \(\dvr\)\nb-linear section.  The second
functor is invariant under polynomial homotopies, stable for
matrices and satisfies excision for all extensions of
\(\resf\)\nb-algebras.  To be more precise, the first functor above
is the composite of the functor defined
in~\cite{Cortinas-Meyer-Mukherjee:NAHA} with the dissection functor
from the category of complete bornological \(\dvf\)\nb-vector spaces
to the category of inductive systems of Banach \(\dvf\)\nb-vector
spaces.  This composite is already used
in~\cite{Meyer-Mukherjee:HA}.  The main result
in~\cite{Meyer-Mukherjee:HA} is the following:

\begin{theorem}[\cite{Meyer-Mukherjee:HA}*{Theorem~5.9}]
  \label{thm:pro-dagger_lifting}
  Let~\(A\) be an \(\resf\)\nb-algebra.  Let \(D=(D_n)_{n\in\N}\) be
  a projective system of dagger algebras that is fine
  mod~\(\dvgen\).  Let \(\varrho\colon D\to A\) be a
  pro-homomorphism, represented by a coherent family of surjective
  homomorphisms \(\varrho_n\colon D_n \to A\) for \(n\in\N\).
  Assume that \(\ker \varrho = (\ker \varrho_n)_{n\in\N}\) is
  analytically nilpotent.  Then there is a canonical
  quasi-isomorphism \(\HAC(A) \cong \HAC(D)\).
\end{theorem}

In particular, if~\(D\) is a dagger algebra and fine mod~\(\dvgen\),
then \(\HAC(D) \cong \HAC(D/\dvgen D)\).  This result will be the
basis for our study of local cyclic homology.  Here it is crucial to
restrict to dagger algebras that are fine mod~\(\dvgen\).  We do not
know how to compute \(\HAC(D)\) when~\(D\) is a Banach
\(\dvr\)\nb-algebra as in Example~\ref{exa:Banach_algebra}.

\begin{definition}
  \label{definition:HL}
  Let~\(D\) be a dagger algebra and let~\(D'\) be the algebra~\(D\)
  with the compactoid bornology.  The \emph{local cyclic homology
    complex} of~\(D\) is the chain complex
  \(\HLC(D) \defeq \HAC(D')\), viewed as an object in the derived
  category of
  \(\overleftarrow{\mathsf{Ind}(\mathsf{Ban}_\dvf)}\).
\end{definition}

\begin{lemma}
  \label{functoriality:HL}
  \(\HLC\) is a functor on the category of dagger algebras.
\end{lemma}

\begin{proof}
  The compactoid bornology is functorial by
  Lemma~\ref{lem:compactoid_adjoint_functor}, and so is \(\HAC\).
\end{proof}

\begin{theorem}
  \label{the:dependence_mod_p}
  Let \(D_1\) and~\(D_2\) be dagger algebras.  If
  \(D_1/\dvgen D_1 \cong D_2/\dvgen D_2\), then
  \(\HLC(D_1) \cong \HLC(D_2)\).
\end{theorem}

\begin{proof}
  By Lemma~\ref{lem:compactoid_dagger}, the algebras \(D_1'\)
  and~\(D_2'\) are again dagger algebras.  They are fine
  mod~\(\dvgen\) by construction of the compactoid bornology, and
  \(A\defeq D'_1/\dvgen D'_1 \cong D'_2/\dvgen D'_2\).
  Theorem~\ref{thm:pro-dagger_lifting} implies
  \[
    \HLC(D_1)
    \cong \HAC(D_1')
    \cong \HAC(A)
    \cong \HAC(D_2')
    \cong \HLC(D_2).\qedhere
  \]
\end{proof}

Theorem~\ref{the:dependence_mod_p} explains our choice of the
compactoid bornology in the definition of local cyclic homology.
Its crucial features are that it is fine mod~\(\dvgen\) and that it
still gives a dagger algebra.  This ensures that \(\HLC(D)\) for a
dagger algebra~\(D\) depends only on \(D/\dvgen D\).  There are,
however, other bornologies that may work just as well.  For
instance, we could give a dagger algebra the linear growth bornology
generated by the fine bornology.  Some work would, however, be
required to check whether this bornology is complete.  Another
advantage of the compactoid bornology is that it does not use the
multiplication and that it is analogous to the precompact bornology
in the archimedean case, which is used in that setting to define
local cyclic homology (see~\cite{Meyer:HLHA}).

\begin{corollary}
  \label{cor:HL_invariant_dagger}
  Let~\(R\) be a torsionfree \(\dvr\)\nb-algebra and
  let~\(\coma{R}\) be its \(\dvgen\)\nb-adic completion.  Give~\(R\)
  the fine bornology and equip~\(\coma{R}\) with the bornology where
  all subsets are bounded.  There are natural quasi-isomorphisms
  \[
    \HAC(R^\updagger) \cong \HLC(R^\updagger) \cong \HLC(\coma{R}).
  \]
\end{corollary}

\begin{proof}
  The algebra~\(\coma{R}\) is a Banach \(\dvr\)\nb-algebra as in
  Example~\ref{exa:Banach_algebra}, and~\(R^\updagger\) is a dagger
  algebra by construction.  It is nuclear by
  Corollary~\ref{cor:Monsky-Washnitzer-compactoid}.  So
  $\HAC(R^\updagger) = \HLC(R^\updagger)$.  We compute
  \[
    \coma{R}/\dvgen \coma{R}
    \cong R/\dvgen R
    = \ling{R}/\dvgen \ling{R}
    \cong \comb{\ling{R}}/\dvgen \comb{\ling{R}}
    = R^\updagger/\dvgen R^\updagger;
  \]
  the isomorphism \(\ling{R}/\dvgen \ling{R} \cong
  \comb{\ling{R}}/\dvgen \comb{\ling{R}}\) follows because
  \(\ling{R}/\dvgen \ling{R}\)
  as an \(\resf\)\nb-algebra is bornologically complete as a
  \(\dvr\)\nb-algebra (compare
  \cite{Cortinas-Meyer-Mukherjee:NAHA}*{Proposition~2.3.3}).  Now
  Theorem~\ref{the:dependence_mod_p} implies
  $\HLC(R^\updagger) \cong \HLC(\coma{R})$.
\end{proof}

\section{Homotopy invariance, excision and matricial stability}
\label{sec:formal_properties}

We are going to prove that local cyclic homology satisfies some nice
formal properties, namely, homotopy invariance, matricial stability,
excision, and an analogue of Bass' Fundamental Theorem.  These will
be used in Section~\ref{sec:computations} to compute the local
cyclic homology of certain Banach \(\dvr\)\nb-algebras.

\begin{theorem}[Homotopy invariance]
  \label{thm:homotopy_invariance}
  Let~\(D\) be a dagger algebra.  Give \(\coma{\dvr[t]}\) the
  bornology where all subsets are bounded.  There is a
  quasi-isomorphism
  \[
    \HLC(D) \cong \HLC\bigl(D \hot \coma{\dvr[t]}\bigr).
  \]
\end{theorem}

\begin{proof}
  The bornological algebras \(D'\) and \((D \hot \coma{\dvr[t]})'\)
  are both dagger by Proposition~\ref{lem:compactoid_dagger} and
  fine mod~\(\dvgen\).  Then Theorem~\ref{thm:pro-dagger_lifting}
  and \cite{Meyer-Mukherjee:HA}*{Theorem~8.1} imply
  \[
    \HLC(D)
    \cong \HAC(D/\dvgen D)
    \cong \HAC(D/ \dvgen D \otimes \resf[t])
    \cong  \HAC\bigl(\bigl(D \hot \coma{\dvr[t]}\bigr)'\bigr)
    = \HLC\bigl(D \hot \coma{\dvr[t]}\bigr).\qedhere
  \]
\end{proof}

An \emph{elementary homotopy} between two bounded homomorphisms
\(f,g\colon D_1 \rightrightarrows D_2\) is a bounded homomorphism
\(F\colon D_1 \to D_2 \hot \coma{\dvr[t]}\) with
\(\ev_0\circ F = f\) and \(\ev_1\circ F = g\).  Let \emph{homotopy}
be the equivalence relation generated by elementary homotopy.  If
\(f,g\colon D_1 \rightrightarrows D_2\) are homotopic, then
\(\HLC(f) = \HLC(g)\) by Theorem~\ref{thm:homotopy_invariance}.

\begin{corollary}
  \label{cor:Banach_algebra:homotopy_invariance}
  For any Banach \(\dvr\)\nb-algebra~\(D\), there is a natural
  quasi-isomorphism \(\HLC(D) \cong \HLC\bigl(\coma{D[t]}\bigr)\).
  Here \(\coma{D[t]}\) is the Banach \(\dvr\)\nb-algebra of all
  polynomials \(\sum_{n=0}^\infty d_n t^n\) with \(d_n \in D\) for
  all \(n\in\N\) and \(\lim d_n=0\) \(\dvgen\)\nb-adically in~\(D\).
\end{corollary}

\begin{theorem}[Matrix stability]
  \label{thm:matricial_stability}
  Let~\(D\) be a dagger algebra.  For a set~\(\Lambda\), let
  \(M_\Lambda(\dvr)\) be the \(\dvr\)\nb-algebra of finitely
  supported matrices indexed by \(\Lambda \times \Lambda\).  Let
  \(\lambda \in \Lambda\).  The canonical map
  \(i_\lambda \colon D \to D \hot \coma{M_\Lambda(\dvr)}\),
  \(x \mapsto e_{\lambda, \lambda} \otimes x\), induces a
  quasi-isomorphism
  \[
    \HLC(D) \cong \HLC\bigl(D \hot \coma{M_\Lambda(\dvr)}\bigr).
  \]
\end{theorem}

\begin{proof}
  Let \(D/\dvgen D = A\).  We compute
  \[
    \HLC(D)
    \cong \HAC(A)
    \cong \HAC(M_\Lambda(A))
    \cong \HAC\bigl(\bigl(D  \hot \coma{M_\Lambda(\dvr)}\bigr)'\bigr)
    = \HLC\bigl(D \hot \coma{M_\Lambda(\dvr)}\bigr);
  \]
  here the second quasi-isomorphism uses
  \cite{Meyer-Mukherjee:HA}*{Proposition~8.2}, and the third
  quasi-isomorphism uses Theorem~\ref{thm:pro-dagger_lifting} and
  \(D \hot \coma{M_\Lambda(\dvr)}/ \dvgen \cdot \bigl(D \hot
  \coma{M_\Lambda(\dvr)}\bigr) \cong M_\Lambda(A)\).
\end{proof}

\begin{corollary}
  \label{cor:Banach_algebra:stability}
  Let~\(D\) be a Banach \(\dvr\)\nb-algebra.  There is
  quasi-isomorphism
  \(\HLC(D) \cong \HLC\bigl(\coma{M_\Lambda(D)}\bigr)\).  Here
  \(\coma{M_\Lambda(D)}\) is the Banach \(\dvr\)\nb-algebra of all
  matrices \((d_{\lambda,\mu})_{\lambda,\mu\in\Lambda}\) with
  \(d_{\lambda,\mu} \in D\) for all \(\lambda,\mu\in\Lambda\) and
  such that, for each \(k\in\N\), there are only finitely many
  \(\lambda,\mu\in\Lambda\) with
  \(d_{\lambda,\mu} \notin \dvgen^k D\).
\end{corollary}

Next, we prove that local cyclic homology satisfies excision for
\emph{any} extension of dagger algebras; no bounded linear section
is needed.

\begin{theorem}[Excision Theorem]
  \label{thm:excision}
  An extension of dagger algebras
  \(K \overset{i}\into E \overset{p}\onto Q\) induces an exact
  triangle
  \[
    \HLC(K) \overset{i_*}\to \HLC(E) \overset{p_*}\to \HLC(Q)
    \overset{\delta}\to \HLC(K)[-1]
  \]
  in the derived category of the quasi-abelian category
  \(\overleftarrow{\mathsf{Ind}(\mathsf{Ban}_\dvf)}\).
\end{theorem}

\begin{proof}
  The given extension induces an extension of \(\resf\)\nb-algebras
  \[
    K \otimes_\dvr \resf \into E \otimes_\dvr \resf \onto Q
    \otimes_\dvr \resf
  \]
  because~\(Q\) is torsionfree.  This induces a natural exact
  triangle
  \[
    \HAC(K \otimes_\dvr \resf)
    \xrightarrow{i_*} \HAC(E \otimes_\dvr \resf)
    \xrightarrow{p_*} \HAC(Q \otimes_\dvr \resf)
    \xrightarrow{\delta} \HAC(K \otimes_\dvr \resf)[-1]
  \]
  in the derived category by
  \cite{Meyer-Mukherjee:HA}*{Theorem~8.3}.  By
  Theorem~\ref{the:dependence_mod_p}, this exact triangle is
  isomorphic to an exact triangle as in the statement of the
  theorem.
\end{proof}

An exact triangle in a derived category implies a long exact
sequence in homology (compare
\cite{Cortinas-Meyer-Mukherjee:NAHA}*{Theorem~5.1}).

\begin{theorem}[Bass Fundamental Theorem]
  \label{the:Bass_fundamental}
  Let~\(D\) be a dagger algebra.  Then
  \(\HLC\bigl(D \hot \coma{\dvr[t,t^{-1}]}\bigr) \cong \HLC(D) \oplus
  \HLC(D)[1]\); here~\([1]\) means a degree shift.
\end{theorem}

\begin{proof}
  Let \(A = D/\dvgen D\).  Then we compute
  \begin{multline*}
    \HLC\bigl(D \hot \coma{\dvr[t,t^{-1}]}\bigr)
    \cong \HAC\bigl(\bigl(D \hot \coma{\dvr[t,t^{-1}]}\bigr)'\bigr)
    \cong \HAC(A \otimes \resf[t,t^{-1}])
    \\ \cong \HAC(A) \oplus \HAC(A)[1]
    \cong \HAC(D') \oplus \HAC(D')[1]
    \cong \HLC(D) \oplus \HLC(D)[1],
  \end{multline*}
  where
  \(\HAC(A \otimes \resf[t,t^{-1}]) \cong \HAC(A) \oplus
  \HAC(A)[1]\) follows from
  \cite{Meyer-Mukherjee:HA}*{Corollary~8.5}.
\end{proof}

\section{Some computations of local cyclic homology}
\label{sec:computations}

In this section, we compute local cyclic homology for some Banach \(\dvr\)\nb-algebras.

\begin{example}
  \label{ex:Leavitt}
  Let~\(E\) be a directed graph.  Let \(L(\resf, E)\) and
  \(C(\resf, E)\) be its Leavitt and Cohn path algebras
  over~\(\resf\), respectively.  Their lifts \(L(\dvr, E)\) and
  \(C(\dvr, E)\) are torsionfree \(\dvr\)\nb-algebras, which we
  equip with the fine bornology.  Their \(\dvgen\)\nb-adic
  completions \(\coma{L(\dvr, E)}\) and \(\coma{C(\dvr,E)}\) are
  Banach \(\dvr\)\nb-algebras.  The analytic cyclic homology of the
  dagger completions \(L(\dvr, E)^\updagger\) and
  \(C(\dvr, E)^\updagger\) is computed in
  \cite{Cortinas-Meyer-Mukherjee:NAHA}*{Theorem~8.1}.  With the
  compactoid bornology, these dagger completions are still dagger
  algebras by Lemma~\ref{lem:compactoid_dagger}, and still fine
  mod~\(\dvgen\).  
  Corollary~\ref{cor:Monsky-Washnitzer-compactoid} shows directly that
  \(L(\dvr, E)^\updagger\) is nuclear.

  Theorem~\ref{the:dependence_mod_p} implies that the local cyclic
  homology of \(\coma{L(\dvr, E)}\) and \(\coma{C(\dvr,E)}\) is
  naturally isomorphic to the analytic cyclic homology computed
  in~\cite{Cortinas-Meyer-Mukherjee:NAHA}, namely,
  \begin{align*}
    \HLC\bigl(\coma{L(\dvr, E)}\bigr)
    &\cong \HAC(L(\resf, E))
    \cong \coker(N_E) \oplus \ker(N_E)[1],\\
    \HLC\bigl(\coma{C(\dvr, E)}\bigr)
    &\cong \HAC(C(\resf, E))
      \cong \dvf^{(E^0)}.
  \end{align*}
\end{example}

\begin{example}
  \label{ex:Laurent_polynomials}
  Consider the \(\dvr\)\nb-algebra of Laurent polynomials in
  \(n\)~variables
  \(L_n(\dvr) = \dvr[t_1, t_1^{-1}, \dots, t_n, t_n^{-1}]\).  We may
  write this as \(L_{n-1}(\dvr) \otimes \dvr[t_n, t_n^{-1}]\).
  Taking \(\dvgen\)\nb-adic completions gives
  \[
    \coma{L_n(\dvr)} \cong L_{n-1}(\dvr) \haotimes \coma{\dvr[t,t^{-1}]}.
  \]
  Then Theorem~\ref{the:Bass_fundamental} implies
  \[
    \HLC\Bigl(\coma{L_n(\dvr)}\Bigr)
    \cong \HLC\Bigl(\coma{L_{n-1}(\dvr)}\Bigr) \oplus  \HLC\Bigl(\coma{L_{n-1}(\dvr)}\Bigr)[1]
    \cong \HLC\Bigl(\coma{L_{n-1}(\dvr)}\Bigr) \otimes (\dvf \oplus \dvf[1]).
  \]
  Iterating this and using
  \(\HLC(\coma{L_0(\resf)}) = \HAC(\dvr) = \dvf\), we identify
  \(\HLC(\coma{L_n(\dvr)})\) with the \(\dvf\)\nb-vector space
  \(\Lambda^* (\dvf^n)\), the exterior algebra, with the usual
  \(\Z/2\)-grading and the zero boundary map.  This is also
  isomorphic to Berthelot's rigid cohomology of \(\resf[t,t^{-1}]\).
  A closer inspection shows that the canonical chain maps
  \[
    \HLC(\coma{L_n(\dvr)}) \leftarrow \HLC(L_n(\dvr)^\updagger)
    \to \mathbb{HP}(L_n(\dvr)^\updagger \otimes \dvf)
  \]
  are quasi-isomorphisms; here \(\mathbb{HP}\) denotes the chain
  complex that computes periodic cyclic homology
  (see~\cite{Cortinas-Cuntz-Meyer-Tamme:Nonarchimedean} for its
  definition in the current setting).
\end{example}

\begin{example}
  \label{ex:smooth_curves}
  Let~\(X\) be a smooth affine variety over~\(\resf\) and let
  \(A = \mathcal{O}(X)\) be its coordinate ring.
  Elkik~\cite{Elkik:Solutions} shows that there is a smooth
  \(\dvr\)\nb-algebra~\(R\) with \(R/\dvgen R \cong A\).  Its
  \(\dvgen\)\nb-adic completion~\(\coma{R}\) still satisfies
  \(\coma{R}/\dvgen \coma{R} \cong A\).  We call~\(\coma{R}\) the
  \emph{Tate algebra of~\(X\)}.  This is a Banach
  \(\dvr\)\nb-algebra.  Theorem~\ref{thm:pro-dagger_lifting} implies
  \[
    \HLC(\coma{R})
    \defeq \HAC(\coma{R}')
    \cong \HAC(A).
  \]
  In the situation above, we also know that
  \(\HLC(\coma{R}) = \HLC(R^\updagger) = \HAC(R^\updagger) \cong
  \HAC(A)\).  If~\(X\) has dimension~\(1\), then \(\HA_*(A)\) is the
  rigid cohomology of~\(X\) with coefficients in~\(\dvf\) by
  \cite{Meyer-Mukherjee:HA}*{Corollary~5.6}.
\end{example}

\begin{theorem}
  \label{thm:HL_invariance}
  Let~\(B\) be a Banach \(\dvr\)\nb-algebra and let~\(R\) be a
  \(\dvr\)\nb-subalgebra of~\(B\) such that
  \(R/\dvgen R \cong B/\dvgen B\).  Equip~\(R\) with the fine
  bornology.  Then
  \[
    \HAC(R^\updagger) = \HLC(R^\updagger) \cong \HLC(B).
  \]
  In addition, there is a canonical chain map
  \(\HAC(R^\updagger) \to \mathbb{HP}(R^\updagger \otimes \dvf)\).
\end{theorem}

\begin{proof}
  By Lemma~\ref{lem:compactoid_dagger}, both \(B'\)
  and~\((R^\updagger)'\) are dagger algebras.  Since they are also
  fine mod~\(\dvgen\), Theorem~\ref{the:dependence_mod_p} gives a
  quasi-isomorphism \(\HLC(R^\updagger) \cong \HLC(B)\).  The
  periodic cyclic homology complex
  \(\mathbb{HP}(R^\updagger \otimes \dvf)\) for bornological
  \(\dvf\)\nb-algebras like $R^\updagger \otimes \dvf$ is defined
  in~\cite{Cortinas-Cuntz-Meyer-Tamme:Nonarchimedean}.  There is a
  canonical chain map
  \(\HAC(R^\updagger) \to \mathbb{HP}(R^\updagger \otimes \dvf)\)
  because the chain complex \(\HAC(R^\updagger)\) is defined as a
  subcomplex of \(\mathbb{HP}(R^\updagger \otimes \dvf)\).
\end{proof}

Let us return to the situation of Example~\ref{ex:smooth_curves}.
The periodic cyclic homology \(\HP_*(R^\updagger \otimes \dvf)\) is
identified in~\cite{Cortinas-Cuntz-Meyer-Tamme:Nonarchimedean} with
the rigid cohomology of the underlying affine variety
over~\(\resf\).  The computation of \(\HAC(A)\) in
\cite{Meyer-Mukherjee:HA}*{Corollary~5.6} shows that the canonical
chain map \(\HAC(A^\updagger) \to \mathbb{HP}(A^\updagger \otimes \dvf)\) in
Theorem~\ref{thm:HL_invariance} is a quasi-isomorphism for
\(1\)\nb-dimensional varieties.
Example~\ref{ex:Laurent_polynomials} shows that the same happens for
some varieties of higher dimension.  It is unclear, however, how to
generalise these results to other smooth commutative
\(\resf\)\nb-algebras.

One could hope that the canonical chain map
\(\HAC(A^\updagger) \to \mathbb{HP}(A^\updagger \otimes \dvf)\) is an isomorphism
when the algebra~\(A^\updagger\) is ``smooth'' in the sense
that~\(A^\updagger\) has a resolution by projective
\(A^\updagger\)\nb-bimodules of finite length.  The counterexample
in the following proposition shows that this cannot work without
extra assumptions.

\begin{proposition}
  Let \(D \defeq \dvr[t]\) be the coordinate ring of the affine
  plane.  Then~\(\coma{D}'\), the \(\dvgen\)\nb-adic completion
  of~\(D\) equipped with the compactoid bornology, is quasi-free,
  but \(\HA_*(\coma{D}') \not\cong \HP_*(\coma{D}' \otimes \dvf)\).
\end{proposition}

\begin{proof}
  Define \(\Omega^1(D)\) as
  in~\cite{Cuntz-Quillen:Algebra_extensions} and let
  \(\Omega^1(\coma{D}') = \ker\bigl(\coma{D}' \hot \coma{D}' \to
  \coma{D}'\bigr)\).  It is easy to see that the map
  \[
    D \otimes D \to \Omega^1(D), \quad f_1 \otimes f_2 \mapsto f_1
    \cdot \diff t \cdot f_2
  \]
  is an isomorphism of \(D\)\nb-bimodules.  Since~\(D\) is
  torsionfree, the \(\dvgen\)\nb-adic completions form an extension
  \(\coma{\Omega^1(D)} \into \coma{D\otimes D} \onto \coma{D}\).
  Therefore,
  \[
    \Omega^1\bigl(\coma{D}\bigr)
    \cong \coma{\Omega^1(D)}
    \cong \coma{D \otimes D}
    \cong \coma{D} \haotimes \coma{D}.
  \]
  Then
  \(\coma{\Omega^1(D)}' \into \coma{D\otimes D}' \onto \coma{D}'\)
  is an extension by Proposition~\ref{prop:exactness}.
  Proposition~\ref{prop:compactoid_tensor-product} shows that
  \(\coma{D\otimes D}' \cong \coma{D}' \hot \coma{D}'\).  This
  implies an isomorphism
  \[
    \Omega^1(\coma{D}') \cong \coma{D}' \hot \coma{D}'
  \]
  of \(\coma{D}'\)\nb-bimodules.  So \(\Omega^1(\coma{D}')\) is a
  projective \(\coma{D}'\)\nb-bimodule.  Then~\(\coma{D}'\) is
  quasi-free.

  Now let \(\underline{\coma{D}} \defeq \coma{D}' \otimes \dvf\).
  This is again quasi-free.  Since~\(\dvf\) has characteristic zero,
  \(\mathbb{HP}(\underline{\coma{D}}) \cong
  X(\underline{\coma{D}})\).  The homology of this chain complex is
  the algebraic de Rham cohomology of \(\underline{\coma{D}}\).  It
  is well known that this differs from the Monsky--Washnitzer or
  rigid cohomology of the plane
  (see~\cite{Monsky-Washnitzer:Formal}).  At the same time,
  \[
    \HAC\bigl(\coma{D}'\bigr)
    = \HLC\bigl(\coma{D}\bigr)
    \cong \HLC(D^\updagger)
    = \HAC(D^\updagger)
    \cong \HAC(\resf[t])
  \]
  by Theorem~\ref{the:dependence_mod_p}.  This agrees with the rigid
  cohomology of the affine plane.  As a result,
  \(\HA_*(\coma{D}') \not\cong \HP_*(\coma{D}' \otimes \dvf)\).
\end{proof}

\end{document}